\documentclass[12pt]{amsart}
\usepackage{latexsym,amssymb,amsmath}
\usepackage[usenames]{color}
\usepackage{xcolor}
\usepackage{tikz-cd}
\usepackage{mathtools}
\usepackage{verbatim} 
\usepackage[all]{xy}
\usepackage[latin1]{inputenc}
\usepackage[colorlinks=true]{hyperref}
\usepackage{thmtools}

\numberwithin{equation}{section}
\newtheorem*{convention}{Convention}
\newtheorem{theorem}{Theorem}[section]
\newtheorem{lemma}[theorem]{Lemma}
\newtheorem{proposition}[theorem]{Proposition}
\newtheorem{corollary}[theorem]{Corollary}

\newtheorem{theoremx}{Theorem}

\theoremstyle{definition}

\theoremstyle{definition}
\newtheorem{definition}[theorem]{Definition} 
 
\newtheorem{remark}[theorem]{Remark}
\newtheorem{example}[theorem]{Example}

\newtheorem{notation}[theorem]{Notation}

\newcommand{\NN}{\mathbb{N}}

\newcommand{\QQ}{\mathbb{Q}}

\newcommand{\KK}{\mathbb{K}}

\DeclareMathOperator{\depth}{{depth}}

\DeclareMathOperator{\Spec}{{Spec}}

\DeclareMathOperator{\Ass}{{Ass}}
\DeclareMathOperator{\Char}{{Char}}
\DeclareMathOperator{\Min}{{Min}}

\DeclareMathOperator{\Proj}{{Proj}}

\DeclareMathOperator{\reg}{{reg}}

\DeclareMathOperator{\Ann}{{Ann}}

\DeclareMathOperator{\aaa}{{a}}
\DeclareMathOperator{\bb}{{b}}
\DeclareMathOperator{\cc}{{c}}

\DeclareMathOperator{\e}{{e}}

\newcommand{\m}{\mathfrak{m}}
\newcommand{\cF}{\mathcal{F}}
\newcommand{\p}{\mathfrak{p}}
\newcommand{\q}{\mathfrak{q}}
\newcommand{\cA}{\mathcal{A}}
\newcommand{\cV}{\mathcal{V}}

\begin{document}


\title[Regularity index of the generalized minimum distance function]{Regularity index of the generalized minimum distance function}  


\author[Carlos Espinosa-Vald\'ez]{Carlos Espinosa-Vald\'ez$^{\aaa}$}
\address{Carlos Espinosa-Vald\'ez\\
Centro de Investigaci\'on en Matem\'aticas\\ Guanajuato, Gto., M\'exico.
}
\thanks{$^{\aaa}$ The first author was supported by CONACyT Grant 857813.}
\email{carlos.espinosa@cimat.mx }

\author[Luis N\'u\~nez-Betancourt]{Luis N\'u\~nez-Betancourt$^{\bb}$}
\address{Luis N\'u\~nez-Betancourt \\ Centro de Investigaci\'on en Matem\'aticas\\ Guanajuato, Gto., M\'exico.}
\thanks{$^{\bb}$ The second author was partially supported by CONACyT Grant 284598,  C\'atedras Marcos Moshinsky, and SNI, Mexico.}
\email{luisnub@cimat.mx}

\author[Yuriko Pitones]{Yuriko Pitones$^{\cc}$}
\address{Yuriko Pitones\\ 
Universidad Aut\'onoma Metropolitana, Unidad Iztapalapa.}
\thanks{$^{\cc}$ The third author was partially supported by SNI, Mexico.}
\email{ypitones@xanum.uam.mx }


\maketitle 

\parindent=8mm

\begin{abstract}
We show that the generalized minimum distance function is non-increasing as the degree varies for reduced standard graded algebras over a field. This allows us to define its regularity index and its stabilization value.  The stabilization value is computed  for every cases. We study how the regularity index varies as the number of polynomial increases,  and use this to give bounds for it.
\end{abstract}

\section{Introduction}\label{Intro}

The minimum distance of the Reed-Muller codes is a parameter that 
is related with the error correction conditions \cite{G, Geil_2012}. Mart\'inez-Bernal, Pitones and Villarreal \cite{MBPV} introduced an analogue algebraic invariant for a homogeneous ideal in a standard graded polynomial ring, that is used to compute the minimum distance for certain families of codes. In this manuscript, we study the generalized Hamming 
distance and its algebraic version \cite{GSMBVV}, 
however, determining this distance is a difficult problem. Wei \cite{Wei91} studied the generalized Hamming distance and his results imply that the generalized Hamming distance completely characterizes the behavior of the code.


The $\ell$-th generalized Hamming distance of a $\KK$-linear code $C$ is 
$$
\delta_{\ell}(C)=\min_{D\subseteq C\; \&\; \dim_\KK D=\ell}^{}  \mid \chi(D)\mid,
$$
where $\chi(D)=\{i\mid \exists (v_1,\ldots, v_n)\in D, v_i\neq 0\}$.

With this motivation in mind, the generalized minimum distance function \cite{GSMBVV} 
of a homogeneous ideal $I\subseteq S=\KK[x_1,\ldots,x_n]$ is defined by,
$$
\delta_I(t,\ell)=\left\{\begin{array}{ll}\e(S/I)-\max\{\e(S/(I,F))\vert\,
F\in\mathcal{F}_{t,\ell}\}&\mbox{if }\mathcal{F}_{t,\ell}\neq\varnothing,\\
\e(S/I)&\mbox{if\ }\mathcal{F}_{t,\ell}=\varnothing.
\end{array}\right.
$$
Where,
$$\mathcal{F}_{t,\ell}=\{F=\{f_1,\ldots, f_\ell\}\subseteq S_t \mid f_1,\ldots, f_{\ell} \text{ are } \mathbb{K}\text{-linearly independent in } S/I, (I\colon (F))\neq I\}.$$

We note that this numerical invariant depends only of the structure of $R=S/I$ as a graded algebra, and not on its presentation (see Definition \ref{Def MD Algebras}). We chose the approach of finitely generated algebras to obtain results that involves the geometry of $\Proj(R)$ without considering the embedding into a projective space.

If $\ell=1$, several properties of the behavior of the minimum distance function have been obtained 
\cite{MBPV,NBPV,MBPV-CI}. In particular, it is known that $\delta_R (t,1)$ eventually stabilizes for 
radical ideals. Furthermore, there are bounds of its regularity index, which is the value where 
$\delta_R (t,1)$ stabilizes \cite{NBPV}.

If there exists a linear form that is a nonzero divisor for $R$, then $\delta_R(t,\ell)$ stabilizes
for $t\gg 0$ \cite{GSMBVV}. We generalized this result for reduced graded algebras. 

\begin{theoremx}[{\autoref{ThmStabilization}}]\label{MainStab}
If $R$ is a reduced graded algebra,  then $\delta_R (t,\ell)\geq \delta_R (t+1,\ell)$.
Consequently, $\delta_R (t,\ell)$ stabilizes for $t\gg 0$.
\end{theoremx}

As a consequence of \autoref{MainStab}, one can study the value where the generalized minimum distance stabilizes, denoted by $s_{R}$, and its regularity index, denoted by $r_{R}$.
In general, it is a difficult problem  to compute these values, even if  $\KK$ is a finite field.  If $\ell= 1$, this coincides with computing the minimum distance for linear codes. 
 In \autoref{ThmStabValue}, we compute the stabilization values of $\delta_{R}(t,\ell)$ and study the behavior of the regularity index with respect to the growth of $\ell$.

\begin{theoremx}[\autoref{ThmIneqRI}]\label{MainReg}
If $R$ is a reduced graded algebra with $\depth(R)\geq 2$,  then $r_{R}(\ell+1)\leq r_{R}(\ell) +1$.
\end{theoremx}

As a consequence of \autoref{MainReg}, we obtain the following bounds for the regularity index in terms of algebraic invariants.

\begin{theoremx}\label{MainRegBounds}
Let $R$ be a standard graded $\KK$-algebra.
We have the following  bounds for the regularity index 
\begin{enumerate}
\item If $R$ is a Stanley-Reisner ring corresponding to either a shellable or a Gorenstein simplicial complex,
then $
r_{R}(\ell)\leq \reg(R)+\ell-1$ {\rm (\autoref{CorBoundRegSR})}.
\item If $R$ is Gorenstein and  $F$-pure,
then $
r_{R}(\ell)\leq \reg(R)+\ell-1$ {\rm (\autoref{CorBoundRegGor})}.
\item If $R$ is an Stanley-Reisner ring and  $\Proj(R)$ is connected, then
$r_{R}(\ell)\leq \dim(R)+\ell-1$ {\rm (\autoref{ThmBoundSR})}.
\item  If $R$ is $F$-pure, $\KK$ is a separably closed field, and  $\Proj(R)$ is connected, then
$r_{R}(\ell)\leq \dim(R)+\ell-1$ {\rm (\autoref{ThmBoundFpure})}.
\end{enumerate}
\end{theoremx}

\begin{convention}\rm 
Throughout this manuscript we assume that $\KK$ is a field and $R$ is a standard graded $\KK$-algebra of dimension $d$. For an $R$-module $M$, $\e(M)$ denotes the Hilbert-Samuel multiplicity of $M$ as an $R$-module. For the multiplicity with respect to another $\KK$-algebra, $T$, we write $\e_T(M)$.
\end{convention}

\section{Preliminaries}\label{prelim-section}

In this section we recall some well known notions and results that are needed throughout this manuscript.

Let $R$ be a standard graded $\KK$-algebra. Given a graded $R$-module, $M$, with $\dim(M)=\theta$, 
its  {\it Hilbert function}, denoted by $h_M$, is given by
$$
h_M(t)=\dim_\KK[M]_{t}.
$$ 
By a classical theorem of Hilbert there is a unique polynomial of degree $\theta-1$ that agrees with $h_M (t)$ for $t\gg 0$. 
Then, there is a polynomial of degree $\theta$ that agrees with $\sum^t_{i=0} h_M(i)$.

The {\it Hilbert-Samuel multiplicity of\/} $M$ is the positive integer defined by,
$$
\e_R(M)=  \lim\limits_{t\rightarrow\infty}\frac{\theta!\sum^t_{i=0} h_M(i)}{t^{\theta}}.
$$
If the ring $R$ is clear from the context, we just write $\e(M)$.  

Given $t,\ell\in \mathbb{N}_+$ we define the following set, 
$$
\mathcal{F}_{t,\ell}=\{F=\{f_1,\ldots, f_\ell\}\subseteq R_t \mid f_1,\ldots, f_{\ell} \text{ are } \mathbb{K}\text{-linearly independent and } \Ann_{R}(F)\neq 0\},$$
where $\Ann_{R}(F)$ is the annihilator of $F$ as R-module.

\begin{definition}\label{Def MD Algebras}
The function $\delta:\mathbb{N}_{+}\times \mathbb{N}_+ \rightarrow \mathbb{Z}$ given by
$$
\delta_{R}(t,\ell)=\left\{\begin{array}{ll}\e(R)-\max\{\e(R/(F))\vert\,
F\in\mathcal{F}_{t,\ell}\}&\mbox{if }\mathcal{F}_{t,\ell}\neq\varnothing,\\
\e(R)&\mbox{if\ }\mathcal{F}_{t,\ell}=\varnothing,
\end{array}\right.
$$
is called the {\it generalized minimum distance function} of $R$, or simply the {\it GMD function} of $R$. 
\end{definition}

Let $R$ be a commutative Noetherian ring with identity and let $I$ be a homogeneous ideal generated by the forms $f_1,\ldots,f_{\ell}\in R$. Consider the \v{C}ech complex, \v{C}$^{\star}(\bar{f};R)$:
$$
 0\rightarrow R \rightarrow \bigoplus_{i}R_{f_i} \rightarrow \bigoplus_{i,j}R_{f_{i}f_{j}}
\rightarrow \cdots \rightarrow R_{f_1,\ldots,f_{\ell}}\rightarrow 0.
$$ 
where \v{C}$^{i}(\bar{f};R)=\bigoplus_{1\leq j_1\leq\ldots \leq j_i\leq \ell}R_{f_{j_1},\ldots,f_{j_i}}$ and the homomorphism in every summand is a localization map with appropriate sign. 

\begin{definition}
Let $M$ be a graded $R$-modue. The {\it $i$-th local cohomology of $M$ with support in $I$} is defined as 
\begin{center}
$H_{I}^{i}(M)= H^{i}($\v{C}$^{\star}(\bar{f};R)\otimes_{R} M)$.
\end{center}
\end{definition}

\begin{remark}\rm
Since $M$ is a graded $R$-module and $I$ is homogeneous, we have that the local cohomology module $H_{I}^{i}(M)$ is graded. 
\end{remark}

\begin{definition}
Let $M$ be an $R$-module with dimension $\theta$.  The {\it $a_i$-invariants}, $a_i(M)$, for $i=0,\ldots,\theta$ are defined as follows.
If $H_{\m}^{i}(M)\neq 0$, 
\begin{center}
$a_i(M)=\max\{\alpha \mid H_{\m}^{i}(M)_{\alpha}\neq 0\}$,  
\end{center} 
for $0\leq i \leq \theta$, where $H_{\m}^{i}(M)$ denotes the local cohomology module with support in the maximal ideal $\m$.  
If $H_{\m}^{i}(M)= 0$, we set $a_i(M)=-\infty.$

If $\theta=\dim(M)$, then $a_\theta(M)$, is often just called the {\it $a$-invariant} of $M$. 
\end{definition}

The $a$-invariant, is a classical invariant  \cite{GW1}, and is closely related to the Castelnuovo--Mumford regularity.

\begin{definition}
Let $M$ be a finitely generated $R$-module. The {\it Castelnuovo--Mumford regularity of $M$}, $\reg(M)$, is defined as
$$
\reg(M)=\max\{a_i(M)+i \mid 1\leq i\leq d\}.
$$ 
\end{definition}  

\begin{remark}
The Castelnuovo--Mumford regularity can also be defined in terms of the Betti numbers of $M$, taken from a minimal graded free resolution of $M$ as an 
$R$-module, that is, ${\rm reg}(M)=\max\{j-i\vert\,\beta_{ij}\neq 0\}.$
\end{remark}

\begin{definition}
Suppose that $\Char(\KK)=p> 0$. 
Let  $F:R\to R$ be the Frobenius map.
We say that $R$ is $F$-pure if for every $R$-module $M$, we have that
$$
\xymatrix{
M\otimes_R R\ar[rr]^{1_M\otimes_R F}
&&
M\otimes_R R
           }
$$
is injective.
We say that $R$ is $F$-finite if $R$ is finitely generated as $R^p$-module.
\end{definition}

\begin{definition}
Suppose that $R$ has prime characteristic $p$. 
The Frobenius  map $F:R \rightarrow R$  is defined by  $r\mapsto r^p$.
\end{definition}

\begin{remark}
 If $R$ is reduced, $R^{1/p^{e}}$ the ring of the $p^{e}$-th roots of $R$ is well defined, and $R\subseteq R^{1/p^{e}}$.
\end{remark}

A ring $R$ is called {\it unmixed} if all its associated primes have the same height, in other case $R$ is {\it mixed}.

About the asymptotic behavior of the GMD function of $R$, we have the following result. 


\begin{theorem}[{\cite[Theorem~3.9]{CSTVPV}}]
Suppose that $R$ is unmixed. Let $t\geq 1$ and $\ell\geq 1$ be integers. The following hold:
\begin{enumerate}
\item[{\rm (i)}] $\delta_{R}(t,\ell)\leq \delta_{R}(t,\ell +1)$.
\item[{\rm (ii)}] If there is $h\in R_1$ regular on $R$, then $\delta_{R}(t,\ell)\geq \delta_{R}(t+1,\ell)\geq 1$.
\end{enumerate}
\end{theorem}
\section{Asymptotic behavior of the GMD function}\label{section-mindis}

In this section we prove that the generalized minimum distance function $\delta_{R}(t,\ell)$ is a non-increasing function at the degree.
Then, the notion of stabilization value of $\delta_{R}(t,\ell)$ and its regularity index are well defined. We start this section establishing notation.

\begin{notation}
Let $X\subseteq R$ any subset and $(X)$ its generated ideal, we set
\begin{align*}
\cA(R)&=\{\p\in\Min(R)\; | \; \dim(R)=\dim(R/\p)\},\\
\cV(X)&=\{\p\in\Spec(R)\; |\; (X)\subseteq \p\}.
\end{align*}

In addition, $ \lambda_{R_\p}(M_\p)$ denotes the length of $M_\p$ as $R_\p$-module. 
\end{notation}

\begin{remark}[Additivity Formula]\label{RemAdditivity}
For any finitely generated $R$-module $M$, we have 
$$
\e(M)=\sum_{\p\in\cA(R)} \lambda_{R_\p}(M_\p)\e(R/\p).
$$
In particular, if $R$ is a reduced algebra, then
$$
\e(R)=\sum_{\p\in\cA(R)}  \e(R/\p).
$$

\end{remark}

The following lemmas are needed to prove that the GMD function stabilizes. 

\begin{lemma}\label{LemmaAdditivity} 
Let $F\in\mathcal{F}_{t,\ell}$ such that $\dim(R/(F))=\dim(R)$. Then,
\[
\e(R/(F))=\sum_{\p\in \mathcal{A}(R)\cap\mathcal{V}(F)}\e(R/\p).
\]
\end{lemma}

\begin{proof}
Using \autoref{RemAdditivity} we have
\[
\e(R/(F))=\sum_{\p\in \mathcal{A}(R)}\lambda_{R_{\p}}(R_{\p}/(F)R_{\p})\e(R/\p).
\]

Fix $\p\in\mathcal{A}(R)$. Since $R$ is reduced and $R\twoheadrightarrow R_{\p}/(F)R_{\p}$, we have that
\[
1=\lambda_{R_{\p}}(R_{\p})\geq \lambda_{R_{\p}}(R_{\p}/(F)R_{\p})\geq 0.
\]
If $(F)\nsubseteq\p$, then $(F)R_{\p}=R_{\p}$, and so, $\lambda_{R_{\p}}(R_{\p}/(F)R_{\p})=0$. If $(F)\subseteq\p$, then $\lambda_{R_{\p}}(R_{\p}/(F)R_{\p})>0$, and so
 $\lambda_{R_{\p}}(R_{\p}/(F)R_{\p})=1$. Therefore,
\[
\e(R/(F))=\sum_{\p\in \mathcal{A}(R)\cap\mathcal{V}(F)}\e(R/\p).
\]
\end{proof}

\begin{lemma}\label{dis.ig} 
Let $y$ be a variable over $R$. Set $\tilde{R}= R\otimes_{\mathbb{K}}\mathbb{K}(y)$. Then, $\delta_{R}(t,\ell)=\delta_{\tilde{R}}(t,\ell)$.
\end{lemma}

\begin{proof}
By construction of $\tilde{R}$ it follows that $\delta_{R}(t,\ell)\geq \delta_{\tilde{R}}(t,\ell)$.
Let $\p_{1},\ldots,\p_{a}$ be the minimal primes  of $R$. Then, $\tilde{\p}_{1},\ldots,\tilde{\p}_{a}$ are the minimal primes of $\tilde{R}$.
 Let $G\in\mathcal{F}_{t,\ell}$ such that
\[
\delta_{\tilde{R}}(t,\ell)=\e(\tilde{R})-\e(\tilde{R}/(G)).
\]
Then, there exist $\tilde{\p}_{1},\ldots,\tilde{\p}_{b}\in\Ass_{\tilde{R}}(\tilde{R})$ with $b<a$ such that
$G\subseteq\tilde{\p}_{1}\cap\cdots\cap\tilde{\p}_{b}$. Thus,
\[
\dim_{\mathbb{K}}\left[\p_{1}\cap\cdots\cap\p_{b} \right]_{t}
=\dim_{\mathbb{K}(y)}\left[\tilde{\p}_{1}\cap\cdots\cap\tilde{\p}_{b}\right]_{t}
\geq \ell.
\]
Therefore, $\delta_{R}(t,\ell)\leq \delta_{\tilde{R}}(t,\ell)$.
\end{proof}

We are ready to prove one of our main results.

\begin{theorem}\label{ThmStabilization}
If $R$ is a reduced graded algebra,  then $\delta_R (t,\ell)\geq \delta_{R}(t+1,\ell)$.
As a consequence, $\delta_R (t,\ell)$ stabilizes for $t\gg 0$.
\end{theorem}

\begin{proof}
If $\mathcal{F}_{t,\ell}=\varnothing$, then $\delta_R (t,\ell) =\e(R)$. 
Then by definition,
$\delta_R (t+1,\ell)\leq \e(R)=\delta_{R} (t,\ell)$. 

We now assume that $\mathcal{F}_{t,\ell}\neq \varnothing$ and take $F=\{f_1,\ldots,f_\ell\}\in \mathcal{F}_{t,\ell}$ such that
$$
\delta_{R}(t,\ell)=\e(R)-\e(R/(F)).
$$
Let $\p_{1},\ldots,\p_{a}$ be the minimal primes of $R$ with $\dim(R)=\dim(R/\p_i)$ containing $F$. Then, by  \autoref{LemmaAdditivity},
$$
\e(R/(F))=\sum^a_{i=1} \e(R/\p_i).
$$
Since $F\in \p_{1}\cap\ldots\cap\p_{a}$ and the $f_1,\ldots,f_\ell$ are linearly independent in $R$, we have that
$$
\dim_{\KK}\left[\p_{1}\cap\cdots\cap\p_{a}\right]_{t}\geq \ell.
$$

Let $\tilde{R}$  as in \autoref{dis.ig}. Since $\depth(\tilde{R})=\depth(R)\geq 1$,
by prime avoidance, there exists $g\in [\tilde{R}]_{1}$ which is a nonzero divisor in $\tilde{R}$.
Then, $\tilde{R}\xhookrightarrow{g} \tilde{R}$ is injective. In particular,
$gF=\{gf_1,\ldots, gf_\ell\}\in\mathcal{F}_{t+1,\ell}$.
Therefore,
\begin{align*}
\delta_{\tilde{R}}(t+1,\ell)&\leq e(\tilde{R})-e(\tilde{R}/gF))\\
&\leq e(\tilde{R})-e(\tilde{R}/(F)\tilde{R})\\
&=\delta_{\tilde{R}}(t,\ell).
\end{align*}
Using  \autoref{dis.ig} and the previous inequality we conclude that $\delta_{R}(t+1,\ell)\leq \delta_{R}(t,\ell)$.
\end{proof}

\autoref{ThmStabilization} implies that $\delta_R (t,\ell)$ eventually stabilizes at the degree. This motivates  the following definition.

\begin{definition}
The {\it stabilization value of $R$}, denoted by $s_{R}(\ell)$, is defined by
$$
s_R(\ell)=\lim_{t\to \infty}\delta_{R}(t,\ell).
$$
The {\it regularity index of $R$}, denoted by $r_{R}(\ell)$, is defined by
$$
r_{R}(\ell) = \min\{t\in \NN \mid \delta_R (t,\ell) = s_R(\ell) \}.
$$
\end{definition}

The next lemma is useful to study the stabilization of the GMD function.

\begin{lemma}\label{H-increasing} 
Let $M$ be an $R$-module. If $\depth(M)\geq 1$, then its Hilbert function is 
non-decreasing. If  $\depth(M)\geq 2$, then its Hilbert function
strictly increasing starting at $\alpha=\min\{t\; \mid\; [M]_t\neq 0\}$.
\end{lemma}

\begin{proof}
Since the depth and the Hilbert function are non-affected by field extensions, we can consider
$\mathbb{K}$ as an infinite field.
Let $f\in [R]_{1}$ which is a nonzero divisor on $M$. Consider the following exact sequence
\[
\xymatrix{
0\ar[r] & M(-1)\ar[r]^-{f} & M\ar[r] & \overline{M}\ar[r] & 0,
}
\]
where $\overline{M}=M/fM$. 
Since the Hilbert function is an additivity function it follows that
\[
h_{\overline{M}}(t)=h_{M}(t)-h_{M}(t-1).
\]
This implies that $\dim_\KK [M]_{t-1}\leq \dim_\KK [M]_{t},$ and prove our first claim.

We now assume that $\depth(M)\geq 2$. 
We have that $\depth(\overline{M})\geq 1$. Hence, $h_{\overline{M}}(t)$ is non-decreasing.
Let $g\in [R]_{1}$ which is a nonzero divisor on $\overline{M}$. Consider the following injective sequence
\[
\xymatrix{
0\ar[r] & \overline{M}(-1)\ar[r]^-{g} & \overline{M},
}
\]
this implies that $h_{\overline{M}}(t)\geq h_{\overline{M}}(t-1)$. Set $\alpha=\{t\mid M_t\neq 0\}$.
Since $f$ has degree $1$, the injective homomorphism $\xymatrix{M\ar[r]^-{f} & M}$ is not surjective over $M_t$.
Hence, $\overline{M}_\alpha\neq 0$. This implies that, $\dim_{\mathbb{K}}[\overline{M}]_t\geq 1$ for all $t\geq \alpha$.
Thus,
\[
h_{\overline{M}}(t)=h_{M}(t)-h_{M}(t-1)\geq 1,
\]
for all $t\geq\alpha$. Therefore, $h_{M}(t)>h_{M}(t-1)$ for all $t\geq \alpha$.
\end{proof}

The following theorem is another main result of this work, for every cases we compute the stabilization value of the GMD function.

\begin{theorem}\label{ThmStabValue}
Suppose that $R$ is a reduced graded algebra with minimal primes
$$\q_{1},\ldots, \q_{a}, \p_1,\ldots, \p_b$$ such that $\dim(R/\q_i)=\dim(R)$, 
$\dim(R/\p_i)<\dim(R)$, and $\e(R/\q_{i})\leq \e(R/\q_{j})$ for $i\leq j$.
\begin{enumerate}
\item\label{ThmStabValue-1} If $R$ is mixed and $\dim(R/\p_{j})\geq 2$ for some $j$, then $s_{R}(t,\ell)=0$.
\item\label{ThmStabValue-2} If $R$ is a domain, then $s_{R}(t,\ell)=\e(R)$.
\item\label{ThmStabValue-3} If $R$ is unmixed and $\dim(R)\geq 2$, then $s_{R}(t, \ell)=\e(R/\q_{1})$.
\item\label{ThmStabValue-4} If $R$ is one-dimensional and $\ell\leq \e(R)-\e(R/\q_1)$, then 
$$
s_{R}(t,\ell)=\min\left\{\sum_{i\in \sigma} \e(R/\q_i)\;\mid\; \sigma\subsetneq [a] \text{ and }  \ell \leq \sum_{i\in \sigma} \e(R/\q_i)\right\}.
$$
\item\label{ThmStabValue-5} If $R$ is one-dimensional and $\ell> \e(R)-\e(R/\q_1)$, then $s_{R}(t,\ell)=\e(R)$.
\item\label{ThmStabValue-6} If $R$ is mixed, $\dim(R/\p_{j})=1$ for all $j$, and  $\ell\leq \e_{T}(T)$ with 
$T= R/\bigcap^b_{i=1}\p_i$, then $s_{R}(t,\ell)=0$.
\item\label{ThmStabValue-7} If $R$ is mixed, $\dim(R/\p_{j})=1$ for all $j$,
and $\ell>\e_{T}(T)$ with $T= R/\bigcap^b_{i=1}\p_i$, then  $s_{R}(t,\ell)=\e(R/\q_1)$.
\end{enumerate}
\end{theorem}

\begin{proof}
\hfill\medskip
\begin{itemize}
\item[(1)] Let $N=\q_{1}\cap\cdots\cap \q_{a}$. Then, $\Ass(N)=\{\p_1,\ldots,\p_b\}$. This implies that $\dim(N)\geq 2$. 
Then for some degree $\alpha$, there exists a $\mathbb{K}$-linearly independent set
$$F=\{f_{1},\ldots,f_{\ell}\}\subseteq [N]_{\alpha}$$
such that $\Ann_R (F)\neq 0$. Thus, $F\in\mathcal{F}_{t,\ell}$.
Furthermore, $\dim(R)=\dim(R/(F))$.
Hence $\e(R)=\e(R/(F))$. Therefore, $\delta_{R}(\alpha,\ell)=0$. Since $\delta_{R}(-,\ell)$ is non-increasing
it follows that $\delta_{R}(t,\ell)=0$ for all $t\geq \alpha$.

\item[(2)] Since $R$ is a domain it follows that $\mathcal{F}_{t,\ell}=\varnothing$ for all $\ell,t$. Therefore,
$\delta_{R}(t,\ell)=\e(R)$ for all $\ell,t$.

\item[(3)]
We suppose that $R$ is not a domain, because this case is covered in Part (\ref{ThmStabValue-2}). Since $R$ is unmixed, $\cA(R)=\Ass(R)$.
For all $F\in \cF_{t,\ell}$, there exists $\p\in\Ass(R)$ such that
$F\nsubseteq \p$. 
We take $F\in \cF_{t,\ell}$ that gives the generalized minimum distance function. Then,
$$
\delta_{R}(t,\ell)=\e(R)-\e(R/(F))\geq \e(R/\p)\geq \min\{\e(R/\p)\mid \p\in\Ass(R)\}.
$$
Let $N=\q_{2}\cap\cdots\cap\q_{a}$. Since $\Ass(N)=\{\q_{1}\}$ it follows that $\dim(N)\geq 2$.
Then for some degree $\alpha$, there exists a $\mathbb{K}$-linearly independent set,
$F=\{f_{1},\ldots,f_{\ell}\}\subseteq [N]_{\alpha}$, such that
$\Ann_R (F)\neq 0$. Hence, $F\in\cF_{\alpha,\ell}$. Thus,
\[
\delta_{R}(\alpha,\ell)\leq \e(R)-\e(R/(F))=\e(R/\q_{1}).
\]
Since $\delta_{R}(-,\ell)$ is non-increasing it follows that $\delta_{R}(t,\ell)=e(R/\q_{1})$ for all $t\geq \alpha$.

\item[(4)]
We suppose that $R$ is not a domain because this case is covered in Part (\ref{ThmStabValue-2}).
Observe that $R$ is unmixed.
Let $\sigma\subsetneq [a]$ non-empty. We set $J_{\sigma}=\bigcap_{i\notin\sigma}\q_{i}$. 
Then, $\Ass(J_{\sigma})=\{\q_{i}\mid i\in\sigma\}$, as an $R$-module. Hence $\dim(J_{\sigma})=1$. From the short exact sequence
\begin{equation*}
\xymatrix{
0\ar[r] & J_{\sigma}\ar[r] & R\ar[r] & R/J_\sigma\ar[r] & 0,
}
\end{equation*}
 we have that $\e(J_{\sigma})=\sum_{i\in\sigma} \e(R/\q_i)$.
We note that for every $\sigma$, there exists $\alpha_\sigma\in\NN$ such that $\dim_\KK[J_\sigma]_t=\e(J_{\sigma})$ for $t\geq \alpha_\sigma$.
Let $\alpha=\max\{\alpha_{\sigma}\}$.
If $\ell\leq \e(J_{\sigma})$ and $t\geq \alpha$, there exists a $\KK$-linearly
independent set $F=\{f_1,\ldots,f_\ell\}\subseteq [J_\sigma]_t$. Then,
$$
\e(R)-\e(R/(F)) = \sum_{F\not\in \q_i} \e(R/\q_i)\geq \sum_{J_\sigma\not\subseteq \q_i} \e(R/\q_i)=  \e(J_\sigma).
$$
Hence, $$\delta_R(t,\ell)\geq \min\left\{\e(J_\sigma)\; \mid \; \e(J_\sigma)\geq \ell\right\}.$$

We fix $t\geq \alpha$, and set $\gamma\subsetneq [a]$ such that 
$$\e(J_\gamma)=\min\left\{\e(J_\sigma)\; \mid \; \e(J_\sigma)\geq \ell\right\}$$
and $G=\{g_1,\ldots,g_\ell\}\subseteq [J_\gamma]_t$ is a $\KK$-linearly independent set.
We note that $(G)\not\subseteq J_\gamma \cap \q_i$ for any $i\in \gamma$; otherwise,
$$\dim_\KK [(G)]_t\leq \dim_\KK [J_\gamma \cap \q_i]_t=\e(J_\gamma\cap\q_i)=\e(J_\gamma)-\e(R/\q_i)<\ell,$$ which is not possible by our choice of $\gamma$.
Hence,
$$
\delta_R(t,\ell)\leq \e(R)-\e(R/(F)) = \e(R)-\e(R/J_\gamma)= \min\left\{\e(J_\sigma)\;\mid\; \e(J_\sigma)\geq \ell\right\}
$$
by \autoref{LemmaAdditivity}.
We conclude that  
$$
\delta_R(t,\ell)=\min\left\{\e(J_\sigma)\;\mid\; \e(J_\sigma)\geq \ell\right\}.
$$

\item[(5)] 
Note that in this case $R$ is unmixed.
We show by contradiction that $\cF_{t,\ell}=\varnothing$.
Suppose that there exists a $\KK$-linearly independent set, 
$G=\{g_1,\ldots, g_\ell\}\subseteq [R]_{t\geq 1}$,
such that $\Ann_R (G)\neq 0$.
Then, there exists $i$ such that  $G\subseteq \q_{i}$ by prime avoidance.
Thus, $\Ass(\q_i)=\{\q_{j}\mid j\neq i\}$. Hence $\dim_R(\q_i)=1$, where we consider $\q_i$ as an $R$-module. From the short exact sequence
\begin{equation*}
\xymatrix{
0\ar[r] & \q_i\ar[r] & R\ar[r] & R/\q_i\ar[r] & 0,
}
\end{equation*}
 we have that $\e(\q_i)=\e(R)-\e(R/\q_i)\leq \e(R)-\e(R/\q_1)$.
There exists $\alpha\in\NN$ such that $\dim_\KK[\q_i]_t=\e(R)-\e(R/\q_i)\leq \e(R)-\e(R/\q_1)$ for $t\geq \alpha$.
Since $\depth(\q_i)>0$, we have that the Hilbert function of $\q_i$
is non-decreasing. Then, $\dim_\KK [\q_i]_t\leq \e(R)-\e(R/\q_1)$ for all $t$, this contradicts the fact that 
$G\subseteq R$ is a  $\KK$-linearly independent set, as $\ell>\e(R)-\e(R/\q_1)$.

\item[(6)] Let $N=\q_{1}\cap\cdots\cap \q_{a}$. Then, $\Ass(N)=\{\p_{1},\ldots,\p_{b}\}$.
This implies that $\dim(N)=1$. Hence, there exists $\alpha\in\mathbb{N}$ such that 
$\dim_{\mathbb{K}}[N]_{t}=\e_T(T)$
for all $t\geq \alpha$. Thus, there exists a $\mathbb{K}$-linearly independent set 
$F=\{f_{1},\ldots,f_{\ell}\}\subseteq [N]_{\alpha}$. Furthermore, since $\Ann_R (F)\neq 0$ we have that $F\in\cF_{\alpha,\ell}$. Then,
\[
\delta_{R}(\alpha,\ell)\leq \e(R)-\e(R/(F))=0.
\]
Therefore, $\delta_{R}(t,\ell)=0$ for all $t\geq \alpha$.

\item[(7)] 
Let $N=\q_{1}\cap\cdots\cap \q_{a}$.
As in the previous part, 
$\dim(N)=1$ and $\e_T(N)=\e_T(T)$.
Then, it does not exist a $\KK$-linearly independent homogeneous set $F=\{ f_1,\ldots,f_\ell \}$ in $N$.
Thus for every set $G\in\mathcal{F}_{t,\ell}$, $G\nsubseteq N$ and hence
\[
s_{R}(\ell)\geq \e(R/\q_{1}).
\]

We now proceed by cases. We first assume that $a>1$.
Let $V=\q_{2}\cap\cdots\cap \q_{a}$. Then, $\Ass(V)=\{\q_{1},\p_{1},\ldots,\p_{b}\}$.
Thus, $\dim(V)\geq 2$. This implies that for some degree $\alpha_{1}$, there exists a $\mathbb{K}$-linearly independent set
$F=\{f_{1},\ldots,f_{\ell}\}\subseteq [V]_{\alpha_{1}}$
such that $\Ann_R (F)\neq 0$. Hence $F\in\cF_{\alpha_{1},\ell}$.
Furthermore, observe that $\dim(R)=\dim(R/(F))$. Hence,
\[
\delta_R (\alpha_{1},\ell)\leq \e(R)-\e(R/(F))=\e(R/\q_{1}).
\]
Therefore, $s_{R}(\ell)=\e(R/\q_{1})$.

We now assume that $a=1$.
Let $F=\{f_1,\ldots, f_\ell\}\in\mathcal{F}_{\ell,t}$ for some $t\geq 1$. 
Let $\sigma\subsetneq [b]$ non-empty. We set $J_{\sigma}=\bigcap_{i\notin\sigma}\p_{i}$. 
Then, $\Ass(J_{\sigma})=\{\p_{i}\mid i\in\sigma\}\cup \{\q_1\}$, as an $R$-module. 
If $F\subseteq J_{\sigma}$, then $\e(R/(F))=\e(R/\q_1)$ by \autoref{LemmaAdditivity}.
We have that $J_{\sigma}\cap \q_1=J_\sigma \bigcap N$, and $F\nsubseteq J_{\sigma}\cap \q_1$ 
because $\dim_\KK [J_{\sigma}\cap \q_1]_t\leq \dim_\KK [N]_t<\ell$. Hence, $s_R(\ell)=\e(R/\q_1)$.
\end{itemize}
\end{proof}

\begin{example}
Consider the ring $S=\mathbb{F}_{2}[x,y,z]$ and the ideal $I=(x^3+y^2z,xy+z^{2})$. 
Set $R=S/I$. Using {\it Macaulay2} \cite{M2}, the minimal primes of $R$ are
\[
\p_{1}=(x,z), \quad \p_{2}=(y+z,x+z), \text{ and }\; \p_{3}=(xy+z^{2},x^{2}+y^{2}+xz+yz+z^{2}).
\]
Thus, we have that $\dim(R/\p_{i})=1$ for all $i$, $\e(R)=6$, $\e(R/\p_{1})=\e(R/\p_{2})=1$ 
and $\e(R/\p_{3})=4$. Let $\mu\geq0$. For $\ell=1$, we can take $f=y^{\mu}(y^{3}+x^{2}z)\in \mathcal{F}_{\mu+3,1}$. 
Since $f\in\p_{2}\cap\p_{3}$, it follows that
\[
s_{R}(1)=1.
\]
For $\ell=5$, we have that
\[
s_{R}(5)=5,
\]
which is obtained via $F=\{x^{2+\mu}z,x^{\mu+1}z^{2},x^{\mu}y^{2}z,x^{\mu}yz^{2},x^{\mu}z^{3}\}\in\mathcal{F}_{\mu+3,5}\subseteq\p_{1}$.
This proof \autoref{ThmStabValue} (\ref{ThmStabValue-4}).

Since the Hilbert series of $R$ is
\[
1+3t+5t^{2}+\sum_{i=3}^{\infty}6t^{i}.
\]
For $\ell>7$, we have that $\mathcal{F}_{t,r}=\varnothing$ for all $t$. For any linearly independent set of homogeneous elements 
of degree $t$, $F=\{f_{1},\ldots,f_{6}\}$, we have that there exist 
$a_{i}\in\mathbb{F}_{2}$ such that $y^{3+(t-3)}=\sum_{i} a_{i}f_{i}$. Since $y^{3+(t-3)}\notin\p_{i}$ for all $i$, 
it follows that $\mathcal{F}_{t,6}=\varnothing$. Therefore,
\[
s_{R}(\ell)=\e(R)=6 \text{ for }\ell>5.
\]
This proof \autoref{ThmStabValue} (\ref{ThmStabValue-5}).
\end{example}

\begin{example}
Consider the ring $S=\mathbb{F}_{3}[x,y,z]$ and the ideal $I=(y^{2}-yz,x^{2}y-yz^{2})$.
Set $R=S/I$. Using {\it Macaulay2} \cite{M2}, the minimal primes of $R$ are
\[
\q_{1}=(y), \quad \p_{2}=(y-z,x-z), \text{ and }\; \p_{2}=(y-z,x+z),
\]
and we have that $\dim(R/\q_{1})=2$, $\dim(R/\p_{1})=\dim(R/\p_{2})=1$, $\e(R)=\e(R/\q_{1})=1$ and $\e(R/\p_{1})=\e(R/\p_{2})=0$. For $\ell\in\{1,2\}$ there exists $F\in\mathcal{F}_{t,\ell}$ such that $F\subseteq\q_{1}$. Therefore,
\[
s_{R}(\ell)=0.
\]
This \autoref{ThmStabValue} (\ref{ThmStabValue-6}).

For $\ell>2$, we can find a set $F=\{f_{1},\ldots,f_{\ell}\}\subseteq [R]_{t}$ of linearly independent elements for $t\gg 0$.
However, using the Hilbert series of $\q_{1}$ we have that $F\nsubseteq\q_{1}$. This implies that $\e(R/(F))=0$. Therefore,
\[
s_{R}(\ell)=e(R)=1.
\]
This proof \autoref{ThmStabValue} (\ref{ThmStabValue-7}).
\end{example}

The next proposition shows that the GMD function is non-decreasing when $\ell$ grows. 

\begin{proposition}
Suppose that $R$ is a reduced graded algebra. Then, $\delta_{R}(t,\ell)\leq \delta_{R}(t,\ell+1)$.
\end{proposition}

\begin{proof}
By definition we have that $\delta_{R}(t,\ell)\leq \e(R)$. Suppose that $\cF_{t,\ell+1}=\varnothing$. Then, $\delta_{R}(t,\ell+1)=\e(R)$, and so,
$\delta_{R}(t,\ell)\leq \delta_{R}(t,\ell+1)$. 

Now, suppose $\cF_{t,\ell +1}\neq\varnothing$. We take $F=\{f_1,\ldots,f_{\ell+1}\}$ such that $\delta_{R}(t,\ell+1)=\e(R)-\e(R/(F))$.
Let $F'=\{f_1,\ldots,f_\ell\}$ be a $\mathbb{K}$-linearly independent set. 
We have $0\neq \Ann_R (F)\subseteq \Ann_R (F')$. This implies that $F'\in \cF_{t,\ell}$. Since $\Ann_R(F')\subseteq \Ann_R (F)$, we have that $\e(R/(F))\leq \e(R/(F'))$.
Therefore,
$$
\delta_{R}(t,\ell+1)=\e(R)-\e(R/(F))\geq \e(R)-\e(R/(F'))\geq \delta_{R}(t,\ell).
$$
\end{proof}

We now start with preparation lemmas to study the growth of the regularity index of $\delta_R (t,\ell)$.

\begin{lemma}\label{LemmaRImixed}
Suppose that $R$ is a reduced graded mixed algebra with minimal primes
$$\q_{1},\ldots, \q_{a}, \p_1,\ldots, \p_b$$ such that $\dim(R/\q_i)=\dim(R)$ and  
$\dim(R/\p_i)<\dim(R)$. 
Let $N_{1}=(\cap_{i=1}^{a}\q_{i})$ and $N_{2}=(\cap_{i=1}^{b}\p_{i})$.
If $\dim(N_1)\geq 2$, then 
$r_{R}(\ell)=\min\{t\; | \;\dim_{\mathbb{K}}[N_{1}]_{t}\geq \ell\}$.
\end{lemma}
\begin{proof}
Set $\alpha=\min\{t\; | \;\dim_{\mathbb{K}}[N_{1}]_{t}\geq \ell\}$.
By  \autoref{ThmStabValue} (1),
we have $\delta_{R}(t,\ell)=0$ for all $t\geq\alpha$. This implies that $r_{R}(\ell)\leq\alpha$.

Suppose there exists $F\in\cF_{t,\ell}$ with $t<\alpha$. Then $F\nsubseteq N_{1}$. This implies
that there exists a prime ideal $\q_i$ such that $F\nsubseteq\q_i$. Hence, $\delta_{R}(t,\ell)\geq e(R/\q_i)>0$.
Thus, $r_{R}(\ell)\geq\alpha$. However, if such $F$ does not exist, then by \autoref{H-increasing} we have that
$\dim_{\mathbb{K}}[N_1]_{t}<\ell$ for all $t<\alpha$. Hence, $\cF_{t,\ell}=\varnothing$ for all $t<\alpha$,
and so, $\delta_{R}(t,\ell)=\e(R)>0$. Since the function is non-increasing, it follows that
$\delta_{R}(t,\ell)\leq \e(R)$ for all $t\geq \alpha$. Thus, $r_{R}(\ell)\geq\alpha$.
Therefore, $r_{R}(\ell)=\alpha$.
\end{proof}

\begin{lemma}\label{LemmaRIunmixed}
Suppose that $R$ is an unmixed standard graded algebra such that $d\geq 2$.
Let $\p_1,\ldots,\p_a$ be the minimal primes of $R$. Let $\cA=\{i\; \mid\; \e(R/\p_i)\leq \e(R/\p_j) \hbox{ for every }j\}.$
Then, 
$$
r_{R}(\ell)=\min\{t\; | \;\dim_{\mathbb{K}}[\cap_{j\neq i} \p_j]_{t}\geq \ell \hbox{ for some }i\in \cA\}.
$$
\end{lemma}
\begin{proof}
We set $\theta_{i} =\min\{t\; | \;\dim_{\mathbb{K}}[\bigcap_{j\neq i} \p_j]_{t}\geq \ell \hbox{ for some }i\in \cA\}.$
We fix $s\in\cA$ such that $\theta_{s}=\min\{\theta_{i}\}$.
Then, there exists a set $F=\{f_1,\ldots,f_\ell\}\subseteq [\cap_{j\neq i} \p_j]_{\theta_s}$ of linearly independent elements.
Then,
$$
\delta_{R}(\theta_s,\ell)\leq \e(R)-\e(R/(F))=\e(R/\p_s).
$$
By  \autoref{ThmStabValue} (3), $\delta_{R}(r_{R}(\ell),\ell)=\e(R/\p_s)$. Since $\delta_{R}(t,\ell)$ is a non-decreasing function by \autoref{ThmStabilization}, we have that $r_R(\ell)\leq \theta_s$.

If $t< \theta_s$, we have that do not exists a linearly independent set $G=\{g_1,\ldots,g_\ell\}\subseteq [R]_t$ such that $G\not\subseteq  [\bigcap_{j\neq i} \p_j]_{t}$ for any $i\in\cA$.
Hence,
$$
\e(R)-\e(R/(G))> \e(R/\p_s).
$$
Thus, $\delta_{R}(t,\ell)> \e(R/\p_s)$, and so, $t<r_{R}(\ell)$. We conclude that $r_{R}(\ell)\geq \theta_s$.
\end{proof}

The next result shows that the grow of the regularity index of $\delta_R (t,\ell)$ is at most linear with respect to $\ell$.

\begin{theorem}\label{ThmIneqRI}
Suppose that $\depth(R)\geq 2$. Then, $ r_{R}(\ell+1)\leq r_{R}(\ell)+1$.
\end{theorem}
\begin{proof}
Let $\p_1,\ldots,\p_a$ be the minimal primes of $R$. 
Let $\sigma\subsetneq [a]$ and $N_\sigma=\bigcap_{i\not\in \sigma} \p_i$.
From the short exact sequence
$$
0\to N_\sigma \to R\to R/N_\sigma\to 0,
$$
we obtain the long exact sequence 
$$
0\to H^0_\m(N_\sigma)\to H^0_\m(R)\to H^0_\m(R/N_\sigma)\to 
H^1_\m(N_\sigma)\to H^1_\m(R)\to H^1_\m(R/N_\sigma)\to \cdots .
$$
Since $\depth(R)\geq 2$, we have that $ H^0_\m(R)= H^1_\m (R)=0$.
Since $N_\sigma$ is a radical ideal, we have that $H^0_\m (R/N_\sigma)=0$.
From the long exact sequence, we have that
$H^0_\m(N_\sigma)=H^1_\m(N_\sigma)=0$, whence, $\depth(N_\sigma)\geq 2$.
Then, the Hilbert function of $N$ is strictly increasing once it is positive by  \autoref{H-increasing}.

We set $N$ a intersection of prime ideals such that $r_{R}(\ell)=\min\{t\; \mid \; [N]_t\neq 0\}$
in  \autoref{LemmaRImixed} if $R$ is mixed or in \autoref{LemmaRIunmixed} if $R$ is unmixed.
Then, $\dim_\KK [N]_t\geq \ell$. By  \autoref{H-increasing}, we have that 
$\dim_\KK [N]_{t+1}\geq \ell+1$. Hence, $r_{R}(\ell+1)\leq r_{R}(\ell)+1.$
\end{proof}

From previous bounds of the regularity index when $\ell=1$, we obtain bounds for any $\ell\geq 1$.

\begin{corollary}\label{CorBoundDim}
Suppose that $R$ is either Stanley-Reisner ring or an $F$-pure ring.
If $\depth(R)\geq 2$, then
$$
r_{R}(\ell)\leq \dim(R)+\ell-1.
$$
\end{corollary}
\begin{proof}
We have that
$r_{R}(\ell+1)\leq r_{R}(\ell)+1$,  by \autoref{ThmIneqRI}.
Since $r_{R}(1)\leq \dim(R)$  \cite[Theorem 5.5 \& 5.7]{Ovidius},
we conclude that $r_{R}(\ell)\leq \dim(R)+\ell-1$.
\end{proof}

\begin{corollary}\label{CorBoundRegSR}
Suppose that $R$ is a Stanley-Reisner ring corresponding to either a shellable or a Gorenstein simplicial complex.
Then, 
$$
r_{R}(\ell)\leq \reg(R)+\ell-1.
$$
\end{corollary}
\begin{proof}
We have that
$r_{R}(\ell+1)\leq r_{R}(\ell)+1$, by \autoref{ThmIneqRI}.
Since $r_{R}(1)\leq \reg(R)$   \cite[Theorem 5.5 \& 5.7]{Ovidius},
we conclude that $r_{R}(\ell)\leq \reg(R)+\ell-1$.
\end{proof}

\begin{corollary}\label{CorBoundRegGor}
Suppose that $R$ is a Gorenstein $F$-pure ring.
Then, 
$$
r_{R}(\ell)\leq \reg(R)+\ell-1.
$$
\end{corollary}
\begin{proof}
We have that
$r_{R}(\ell+1)\leq r_{R}(\ell)+1$,  by \autoref{ThmIneqRI}.
Since $r_{R}(1)\leq \reg(R$   \cite[Theorem 5.7]{Ovidius},
we conclude that $r_{R}(\ell)\leq \reg(R)+\ell-1$.
\end{proof}

We now seek to rephrase the condition on depth in \autoref{CorBoundDim} 
to the topology of the projective space associated to $R,$ in certain cases.

\begin{lemma}\label{LemmaSRConnected}
Let $S=\mathbb{K}[x_{1},\ldots,x_{n}]$, $I\subseteq S$ a Stanley-Reisner ideal and $R=S/I$.
Then, $\depth(R)\geq 2$ if and only if $\Proj(R)$ is connected.
\end{lemma}
\begin{proof}
We first assume that $\KK$ has prime characteristic.
We note that neither $\depth(R)$ nor the connectedness of $\Proj(R)$ change after field extensions. 
We may assume that $\KK$ is a separably closed field. 
Let $\widehat{S}$ and $\widehat{R}$ denote the completion of $S$ and $R$ at the maximal homogeneous ideal.
We have that $\depth(R)=\depth(\widehat{R})$. In addition, $\Proj(R)$ is connected if and only if the punctured spectrum of $\widehat{R}$, $\Spec^\circ(\widehat{R})$, is connected.
We have that $\depth(\widehat{R})\geq 2$ if and only if $H^{n-1}_{I}(\widehat{S})=0$ \cite[Theorem 4.3]{LyuVan}, as $F\text{-}\depth(R)=\depth(R)$ for Stanley-Reisner rings.
In addition, $H^{n-1}_{I}(\widehat{S})=0$ if and only if $\Spec^\circ(\widehat{R})$ is connected \cite[III, Corollairie 5.5]{P-S} (see also \cite{HartCD,H-L}).

We now assume that $\KK$ has characteristic zero. 
We note that neither $\depth(R)$ nor the connectedness of $\Proj(R)$ change after field extensions.
We may assume that $\KK=\QQ$.
Then, the result follows from reduction to prime characteristic \cite[Theorem 2.3.5]{HHCharZero}.
\end{proof}

\begin{lemma}\label{LemmaFpureConnected}
Let $S=\mathbb{K}[x_{1},\ldots,x_{n}]$, $I\subseteq S$ a homogeneous ideal and $R=S/I$.
Suppose that $\KK$ is a separably closed field and  $R$ is a $F$-pure ring.
Then, $\depth(R)\geq 0$ if and only if $\Proj(R)$ is connected.
\end{lemma}
\begin{proof}
Let $\widehat{S}$ and $\widehat{R}$ denote the completion of $S$ and $R$ at the maximal homogeneous ideal.
We have that $\depth(R)=\depth(\widehat{R})$. In addition, $\Proj(R)$ is connected if and only if the punctured spectrum of $\widehat{R}$, $\Spec^\circ(\widehat{R})$, is connected.
We have that $\depth(\widehat{R})\geq 2$ if and only if $H^{n-1}_{I}(\widehat{S})=0$ \cite[Theorem 4.3]{LyuVan}, as $F\text{-}\depth(R)=\depth(R)$ for $F$-pure rings.
In addition, $H^{n-1}_{I}(\widehat{S})=0$ if and only if $\Spec^\circ(\widehat{R})$ is connected \cite[III, Corollairie 5.5]{P-S} (see also \cite{HartCD,H-L}).
\end{proof}

We now rephrase \autoref{CorBoundDim} 
in terms of the connectedness of $\Proj(R)$.
\begin{theorem}\label{ThmBoundSR}
Suppose that $R$ is an Stanley-Reisner ring.
If $\Proj(R)$ is connected, then
$$
r_{R}(\ell)\leq \dim(R)+\ell-1.
$$
\end{theorem}
\begin{proof}
This result follows from  \autoref{LemmaSRConnected}, \autoref{LemmaFpureConnected}, and \autoref{CorBoundDim}.
\end{proof}

\begin{theorem}\label{ThmBoundFpure}
Suppose that $R$ is an $F$-pure graded algebra over a separably closed field.
If $\Proj(R)$ is connected, then
$$
r_{R}(\ell)\leq \dim(R)+\ell-1.
$$
\end{theorem}
\begin{proof}
This result follows from  \autoref{LemmaSRConnected}, \autoref{LemmaFpureConnected}, and \autoref{CorBoundDim}.
\end{proof}

\bibliographystyle{alpha}
\bibliography{References}

\newcommand{\etalchar}[1]{$^{#1}$}
\begin{thebibliography}{GSMBVV19}

\bibitem[CST{\etalchar{+}}20]{CSTVPV}
Susan~M. Cooper, Alexandra Seceleanu, Stefan~O. Toh\v{a}neanu, Maria~Vaz Pinto,
  and Rafael~H. Villarreal.
\newblock Generalized minimum distance functions and algebraic invariants of
  geramita ideals.
\newblock {\em Advances in Applied Mathematics}, 112:101940, 2020.

\bibitem[Gei08]{G}
O.~Geil.
\newblock On the second weight of generalized reed-muller codes.
\newblock {\em Designs, Codes and Cryptography}, 48:323--330, 2008.

\bibitem[GS]{M2}
Daniel~R. Grayson and Michael~E. Stillman.
\newblock Macaulay2, a software system for research in algebraic geometry.
\newblock Available at \url{http://www.math.uiuc.edu/Macaulay2/}.

\bibitem[GSMBVV19]{GSMBVV}
Manuel Gonz\'{a}lez-Sarabia, Jos\'{e} Mart\'{\i}nez-Bernal, Rafael~H.
  Villarreal, and Carlos~E. Vivares.
\newblock Generalized minimum distance functions.
\newblock {\em J. Algebraic Combin.}, 50(3):317--346, 2019.

\bibitem[GT13]{Geil_2012}
Olav Geil and Casper Thomsen.
\newblock Weighted reed{\textendash}muller codes revisited.
\newblock {\em Designs, Codes and Cryptography}, 66(1-3):195--220, may 2013.

\bibitem[GW78]{GW1}
Shiro Goto and Keiichi Watanabe.
\newblock On graded rings. {I}.
\newblock {\em J. Math. Soc. Japan}, 30(2):179--213, 1978.

\bibitem[Har68]{HartCD}
Robin Hartshorne.
\newblock Cohomological dimension of algebraic varieties.
\newblock {\em Ann. of Math. (2)}, 88:403--450, 1968.

\bibitem[HH99]{HHCharZero}
Melvin Hochster and Craig Huneke.
\newblock Tight closure in equal characteristic zero.
\newblock 1999.

\bibitem[HL90]{H-L}
Craig Huneke and Gennady Lyubeznik.
\newblock On the vanishing of local cohomology modules.
\newblock {\em Invent. Math.}, 102(1):73--93, 1990.

\bibitem[Lyu06]{LyuVan}
Gennady Lyubeznik.
\newblock On the vanishing of local cohomology in characteristic {$p>0$}.
\newblock {\em Compos. Math.}, 142(1):207--221, 2006.

\bibitem[MBPV17]{MBPV}
Jos{\'e} Mart{\'i}nez-Bernal, Yuriko Pitones, and Rafael~H. Villarreal.
\newblock Minimum distance functions of graded ideals and {R}eed-{M}uller-type
  codes.
\newblock {\em J. Pure Appl. Algebra}, 221(2):251--275, 2017.

\bibitem[MBPV18]{MBPV-CI}
Jos\'{e} Mart\'{\i}nez-Bernal, Yuriko Pitones, and Rafael~H. Villarreal.
\newblock Minimum distance functions of complete intersections.
\newblock {\em J. Algebra Appl.}, 17(11):1850204, 22, 2018.

\bibitem[NBPV18]{NBPV}
Luis N{\'u\~n}ez-Betancourt, Yuriko Pitones, and Rafael~H. Villarreal.
\newblock Footprint and minimum distance functions.
\newblock {\em Commun. Korean Math. Soc.}, 33(1):85--101, 2018.

\bibitem[NBPV21]{Ovidius}
Luis N{\'{u}\~{n}}ez-Betancourt, Yuriko Pitones, and Rafael~H. Villarreal.
\newblock Bounds for the minimum distance function.
\newblock {\em An. \c{S}tiin\c{t}. Univ. ``Ovidius'' Constan\c{t}a Ser. Mat.},
  29(3):229--242, 2021.

\bibitem[PS73]{P-S}
C.~Peskine and L.~Szpiro.
\newblock Dimension projective finie et cohomologie locale. {A}pplications \`a
  la d\'emonstration de conjectures de {M}. {A}uslander, {H}. {B}ass et {A}.
  {G}rothendieck.
\newblock {\em Inst. Hautes \'Etudes Sci. Publ. Math.}, (42):47--119, 1973.

\bibitem[Wei91]{Wei91}
V.K. Wei.
\newblock Generalized hamming weights for linear codes.
\newblock {\em IEEE Transactions on Information Theory}, 37(5):1412--1418,
  1991.

\end{thebibliography}

\end{document}